\documentclass[11pt, reqno]{amsart} \textwidth=14.9cm \oddsidemargin=1cm \evensidemargin=1cm
\usepackage{amsmath,amssymb,amscd,mathrsfs,stmaryrd,wasysym}
\usepackage{amsmath}
\usepackage{amsxtra}
\usepackage{amscd}
\usepackage{amsthm}
\usepackage{amsfonts}
\usepackage{amssymb}
\usepackage{eucal}
\usepackage{color}
\usepackage{graphicx}%
\usepackage{bm}
\newtheorem{theorem}{Theorem}[section]
\newtheorem{lemma}[theorem]{Lemma}
\newtheorem{proposition}[theorem]{Proposition}

\theoremstyle{definition}

\definecolor{A}{rgb}{.75,1,.75}

\numberwithin{equation}{section}

\begin{document}

\title[0-Yokonuma-Hecke algebras]{Representations of $0$-Yokonuma-Hecke algebras}
\author[Weideng Cui]{Weideng Cui}
\address{School of Mathematics, Shandong University, Jinan, Shandong 250100, P.R. China.}
\email{cwdeng@amss.ac.cn}

\begin{abstract}
We give two different approaches to classifying the simple modules of $0$-Yokonuma-Hecke algebras $Y_{r,n}(0)$ over an algebraically closed field of characteristic $p$ such that $p$ does not divide $r.$ Using the isomorphism between the $0$-Yokonuma-Hecke algebra and $0$-Ariki-Koike-Shoji algebra, we in fact give another way to obtain the simple modules of the latter, which was previously studied by Hivert, Novelli and Thibon (Adv. Math. $\bf{205}$ (2006) 504-548). In the appendix, we give the classification of simple modules of the nil Yokonuma-Hecke algebra.
\end{abstract}

%We establish an equivalence between a module category of the affine (resp. cyclotomic) Yokonuma-Hecke algebra (resp. $Y_{r,n}^{\lambda}(q)$) and its suitable counterpart for a direct sum of tensor products of affine Hecke algebras of type $A$ (resp. cyclotomic Hecke algebras). We then develop several applications of this result. The simple modules of affine Yokonuma-Hecke algebras and of their associated cyclotomic Yokonuma-Hecke algebras are classified over an algebraically closed field of characteristic $p=0$ or $(p,r)=1.$ The modular branching rules for these algebras are obtained, and they are further identified with crystal graphs of integrable modules for affine Lie algebras of type $A.$\end{abstract}

%\thanks{{Keywords}: Affine Yokonuma-Hecke algebras; Cyclotomic Yokonuma-Hecke algebras; Modular branching rules; Crystal graphs} \large

\maketitle
\medskip
\section{Introduction}
\subsection{}
The Iwahori-Hecke algebra $\mathcal{H}_{q}(W)$ can be regarded as a $q$-deformation of the group algebra of a finite Coxeter group $W$, which arises in the study of groups with $(B,N)$-pairs and plays an important role in the study of representations of finite groups of Lie type.

When $q=0$, we get the so-called $0$-Hecke algebra $\mathcal{H}_{0}(W)$, whose representation theory is very different from that of $\mathcal{H}_{q}(W)$ with a non-zero parameter $q.$ In [No], Norton classified the irreducible modules and indecomposable projective modules of $\mathcal{H}_{0}(W)$. In [Ca], Carter gave the decomposition numbers of it in type $A.$ $\mathcal{H}_{0}(W)$ has been studied extensively by various people; see [DHT, DY, Fa, He, HNT, Hu, KT, YL] and so on.

\subsection{}
Yokonuma-Hecke algebras were introduced by Yokonuma \cite{Yo} as a centralizer algebra associated to the permutation representation of a finite Chevalley group $G$ with respect to a maximal unipotent subgroup of $G$. By the presentation given by Juyumaya and Kannan \cite{Ju1, Ju2, JuK}, the Yokonuma-Hecke algebra $Y_{r,n}(q)$ (of type $A$) can be regraded as a deformation of the group algebra of the complex reflection group $G(r,1,n),$ which is isomorphic to the wreath product $(\mathbb{Z}/r\mathbb{Z})\wr \mathfrak{S}_{n}$, where $\mathfrak{S}_n$ is the symmetric group. Yokonuma-Hecke algebras have been studied in [ChPA, C, JaPA, Lu] and so on.

The modified Ariki-Koike algebra [SS], as a way of approximating the usual Ariki-Koike algebra, was first defined by Shoji [S] in order to give a Frobenius type formula for the characters of Ariki-Koike algebras. In [HNT], they studied the representation theory of the modified Ariki-Koike algebra when the parameter $q=0$, called $0$-Ariki-Koike-Shoji algebra; they classified its simple modules and projective modules, and described its Cartan invariants and decomposition matrices. Later on, Espinoza and Ryom-Hansen [ER] proved that the Yokonuma-Hecke algebra is isomorphic to the modified Ariki-Koike algebra when the parameter $q$ is invertible.

In this note, we shall consider the particular case $Y_{r,n}(0)$ when $q=0,$ which we call $0$-Yokonuma-Hecke algebras. In fact, we can easily see that we also have an isomorphism between the two algebras mentioned above when $q=0,$ that is, the $0$-Yokonuma-Hecke algebra is isomorphic to the $0$-Ariki-Koike-Shoji algebra with a particular choice of the parameters. In particular, we give two different approaches to classifying the simple modules of $0$-Yokonuma-Hecke algebras $Y_{r,n}(0)$ over an algebraically closed field of characteristic $p$ such that $p$ does not divide $r.$ Using the isomorphism between $0$-Yokonuma-Hecke algebras and $0$-Ariki-Koike-Shoji algebras, we in fact give another way to obtain the simple modules of the latter, which were previously studied in [HNT].

\section{Preliminaries}

\subsection{$0$-Yokonuma-Hecke algebras}
Let $r, n\in \mathbb{N},$ $r, n\geq1,$ and let $\zeta=e^{2\pi i/r}.$ Let $\mathbb{K}$ be an algebraically closed field of characteristic $p\geq 0$ such that $p$ does not divide $r$, which contains $\zeta$ and an arbitrary element $q.$

The Yokonuma-Hecke algebra $Y_{r,n}(q)$ is a $\mathbb{K}$-associative algebra generated by the elements $t_{1},\ldots,t_{n},g_{1},\ldots,g_{n-1}$ satisfying the following relations:
\begin{equation}\label{rel-def-Y1}\begin{array}{rclcl}
t_i^r\hspace*{-7pt}&=&\hspace*{-7pt}1 && \mbox{for all $i=1,\ldots,n$;}\\[0.2em]
t_it_j\hspace*{-7pt}&=&\hspace*{-7pt}t_jt_i &&  \mbox{for all $i,j=1,\ldots,n$;}\\[0.1em]
g_it_j\hspace*{-7pt}&=&\hspace*{-7pt}t_{s_i(j)}g_i && \mbox{for all $i=1,\ldots,n-1$ and $j=1,\ldots,n$;}\\[0.1em]
g_ig_j\hspace*{-7pt}&=&\hspace*{-7pt}g_jg_i && \mbox{for all $i,j=1,\ldots,n-1$ such that $\vert i-j\vert \geq 2$;}\\[0.1em]
g_ig_{i+1}g_i\hspace*{-7pt}&=&\hspace*{-7pt}g_{i+1}g_ig_{i+1} && \mbox{for all $i=1,\ldots,n-2$;}\\[0.1em]
g_{i}^{2}\hspace*{-7pt}&=&\hspace*{-7pt}q+(q-1)e_{i}g_{i} && \mbox{for all $i=1,\ldots,n-1$,}
\end{array}
\end{equation}
where $s_{i}$ is the transposition $(i,i+1)$, and for each $1\leq i\leq n-1$,
$$e_{i} :=\frac{1}{r}\sum\limits_{s=0}^{r-1}t_{i}^{s}t_{i+1}^{-s}.$$

Note that the elements $e_{i}$ are idempotents in $Y_{r,n}(q).$ For each $w\in \mathfrak{S}_{n},$ let $w=s_{i_1}\cdots s_{i_{r}}$ be a reduced expression of $w.$ By Matsumoto's lemma, the element $g_{w} :=g_{i_1}g_{i_2}\cdots g_{i_{r}}$ does not depend on the choice of the reduced expression of $w,$ that is, it is well-defined. Moreover, the following elements
\begin{align}\label{basis-yokonuam-hecke-1}
\{t_{1}^{a_{1}}\cdots t_{n}^{a_{n}}g_{w}\:|\:1\leq a_{1},\ldots,a_{n}\leq r\text{ and }w\in\mathfrak{S}_{n}\}
\end{align}
form a $\mathbb{K}$-basis of $Y_{r,n}(q).$

Let $i, k\in \{1,2,\ldots,n\}$ and set \begin{equation}\label{idempotents}e_{i,k} :=\frac{1}{r}\sum\limits_{s=0}^{r-1}t_{i}^{s}t_{k}^{-s}.\end{equation} Note that $e_{i,k}^{2}=e_{i,k}=e_{k,i},$ and that $e_{i,i+1}=e_{i}.$ It can be easily checked that
\begin{equation}\label{relations-cude}\begin{array}{rclcl}
t_{i}e_{j,k}\hspace*{-7pt}&=&\hspace*{-7pt}e_{j,k}t_{i} && \mbox{for all $i,j,k=1,\ldots,n$,}\\[0.1em]
e_{i,j}e_{k,l}\hspace*{-7pt}&=&\hspace*{-7pt}e_{k,l}e_{i,j} &&  \mbox{for all $i,j,k,l=1,\ldots,n$,}\\[0.1em]
e_ie_{k,l}\hspace*{-7pt}&=&\hspace*{-7pt}e_{s_i(k), s_i(l)}e_i && \mbox{for all $i=1,\ldots,n-1$ and $k,l=1,\ldots,n$,}\\[0.1em]
e_{j,k}g_{i}\hspace*{-7pt}&=&\hspace*{-7pt}g_{i}e_{s_{i}(j),s_{i}(k)} && \mbox{for all $i=1,\ldots,n-1$ and $j,k=1,\ldots,n$.}
\end{array}
\end{equation}
In particular, we have $e_{i}g_{i}=g_{i}e_{i}$ for all $i=1,2,\ldots,n-1.$

Let $\mathcal{T}$ be the commutative subalgebra of $Y_{r,n}(q)$ generated by $t_1,\ldots,t_{n}$ which is isomorphic to the group algebra of $(\mathbb{Z}/r\mathbb{Z})^{n}$ over $\mathbb{K}.$ A character $\chi$ of $\mathcal{T}$ over $\mathbb{K}$ is determined by the choice of $\chi(t_{j})\in \{\zeta_1,\ldots,\zeta_{r}\}$ for $1\leq j\leq n,$ where $\zeta_{i}=\zeta^{i}$ for $1\leq i\leq r.$ Let $\text{Irr}(\mathcal{T})$ denote the set of characters of $\mathcal{T}$ over $\mathbb{K}.$ The symmetric group $\mathfrak{S}_{n}$ acts by permutations on $\mathcal{T}$ and induces an action on $\text{Irr}(\mathcal{T})$ given by $w(\chi)(t_i)=\chi(t_{w^{-1}(i)})$ for $1\leq i\leq n,$ $w\in \mathfrak{S}_{n}$ and $\chi\in \text{Irr}(\mathcal{T}).$

For each $\chi\in \text{Irr}(\mathcal{T})$, let $E_{\chi}$ be the primitive idempotent of $\mathcal{T}$ associated to $\chi,$ that is, $\chi'(E_{\chi})=0$ if $\chi'\neq\chi$ and $\chi(E_{\chi})=1.$ Then $E_{\chi}$ can be explicitly written in terms of the generators as follows:
\begin{equation}\label{idempotents}
E_{\chi}=\prod_{1\leq i\leq n}\bigg(\frac{1}{r}\sum_{0\leq s\leq r-1}\chi(t_{i})^{s}t_{i}^{-s}\bigg).\end{equation}
By definition, we have, for each $1\leq i\leq n,$ $w\in \mathfrak{S}_{n}$ and $\chi\in \text{Irr}(\mathcal{T}),$
\begin{equation}\label{idempotent-relation-chi}
t_{i}E_{\chi}=E_{\chi}t_{i}=\chi(t_{i})E_{\chi}\qquad \text{ and }\qquad g_{w}E_{\chi}=E_{w(\chi)}g_{w}.\end{equation}
Thus, we immediately get that
\begin{equation}\label{idempotent-relation-chi-ti}
t_{i}=\sum_{\chi\in \text{Irr}(\mathcal{T})}\chi(t_{i})E_{\chi}\qquad \text{ for each }1\leq i\leq n.\end{equation}

The following presentation of $Y_{r,n}(q)$ is proved in [ER, Proposition 2] when $q$ is invertible, which in fact is a particular case of unipotent Hecke algebras considered in [Lu, $\S$31.2].
\begin{proposition}{\rm (See [ER, Proposition 2].)}\label{propo-presenta-2}
$Y_{r,n}(q)$ has a second presentation, which is generated by $g_{1},\ldots,g_{n-1}$ and $E_{\chi}$ $(\chi\in \mathrm{Irr}(\mathcal{T}))$ with relations$:$
\begin{equation*}\label{rel-def-Yokonuma2}\begin{array}{rclcl}
\sum\limits_{\chi\in \text{Irr}(\mathcal{T})}E_{\chi}\hspace*{-7pt}&=&\hspace*{-7pt}1; &&\\[0.2em]
E_{\chi'}E_{\chi}\hspace*{-7pt}&=&\hspace*{-7pt}\delta_{\chi', \chi}E_{\chi}; &&  \\[0.1em]
g_iE_{\chi}\hspace*{-7pt}&=&\hspace*{-7pt}E_{s_{i}(\chi)}g_i && \mbox{for all $1\leq i\leq n-1;$}\\[0.1em]
g_ig_j\hspace*{-7pt}&=&\hspace*{-7pt}g_jg_i && \mbox{for all $1\leq i, j\leq n-1$ such that $\vert i-j\vert > 1;$}\\[0.1em]
g_ig_{i+1}g_i\hspace*{-7pt}&=&\hspace*{-7pt}g_{i+1}g_ig_{i+1} && \mbox{for all $1\leq i\leq n-2;$}\\[0.1em]
g_{i}^{2}\hspace*{-7pt}&=&\hspace*{-7pt}q+(q-1)\sum\limits_{\substack{\chi\in \mathrm{Irr}(\mathcal{T})\\s_{i}(\chi)=\chi}}E_{\chi}g_{i} && \mbox{for all $1\leq i\leq n-1.$}
\end{array}
\end{equation*}
\end{proposition}

By \eqref{basis-yokonuam-hecke-1}, we can easily get that the following elements
\begin{align}\label{basis-yokonuam-hecke-2}
\{E_{\chi}g_{w}\:|\:\chi\in \mathrm{Irr}(\mathcal{T})\text{ and }w\in\mathfrak{S}_{n}\}
\end{align}
form a $\mathbb{K}$-basis of $Y_{r,n}(q).$

The following presentation of $Y_{r,n}(q)$ is proved in [ER, Theorem 7] when $q$ is invertible,, which claims that $Y_{r,n}(q)$ is isomorphic to the modified Ariki-Koike algebra defined in [S, Section 3.6] with a particular choice of the parameters $u_{i},$ that is, $u_{i}=\zeta^{i}$ for $1\leq i\leq r.$
\begin{proposition}{\rm (See [ER, Theorem 7].)}\label{propo-presenta-3}
$Y_{r,n}(q)$ has a third presentation, which is generated by $h_{1},\ldots,h_{n-1}$ and $w_{1},\ldots,w_{n}$ with relations$:$
\begin{align*}
w_i^r&=1 \quad \qquad\qquad\qquad\qquad\mbox{for all $1\leq i\leq n;$}\\[0.1em]
w_iw_j&=w_jw_i \qquad\qquad\qquad\quad\hspace{1.4mm}  \mbox{for all $1\leq i, j\leq n;$}\\[0.1em]
h_{i}h_{j}&=h_{j}h_{i}\qquad \quad\qquad\qquad\hspace{2.1mm}\mbox{for all $1\leq i, j\leq n-1$ such that $\vert i-j\vert > 1;$}\\[0.1em]
h_ih_{i+1}h_i&=h_{i+1}h_ih_{i+1}  \qquad \quad\quad \hspace{3mm} \mbox{for all $1\leq i\leq n-2;$}\\[0.1em]
h_{i}w_i&=w_{i+1}h_{i}-\Delta^{-2}\sum_{c_{1}< c_{2}}(\zeta^{c_2}-\zeta^{c_{1}})(q-1)F_{c_1}(w_{i})F_{c_2}(w_{i+1}) \hspace{2.5mm} \mbox{for all $1\leq i\leq n-1;$}\\[0.1em]
h_{i}(w_i+w_{i+1})&=(w_i+w_{i+1})h_{i} \qquad \quad \hspace{2.2mm} \mbox{for all $1\leq i\leq n-1;$}\\[0.1em]
h_{i}w_{l}&=w_{l}h_{i} \qquad \qquad \qquad\quad \hspace{2.0mm} \mbox{for all $l\neq i, i+1$ and $1\leq i\leq n-1;$}\\[0.1em]
h_{i}^{2}&=q+(q-1)h_{i} \qquad \qquad \hspace{-0.5mm}\mbox{for all $1\leq i\leq n-1.$}
\end{align*}
\end{proposition}

For each $w\in \mathfrak{S}_{n},$ let $w=s_{i_1}\cdots s_{i_{r}}$ be a reduced expression of $w.$ By Matsumoto's lemma again, the element $h_{w} :=h_{i_1}h_{i_2}\cdots h_{i_{r}}$ is well-defined. Moreover, It is known by [S] that the following set
\begin{align}\label{basis-yokonuam-hecke-3}
\{w_{1}^{k_{1}}\cdots w_{n}^{k_{n}}h_{w}\:|\:1\leq k_{1},\ldots,k_{n}\leq r\text{ and }w\in\mathfrak{S}_{n}\}
\end{align}
gives rise to a $\mathbb{K}$-basis of $Y_{r,n}(q).$

For each $\chi\in \text{Irr}(\mathcal{T})$, since $\chi$ is completely determined by the values $(\chi(t_{1}),\ldots,\chi(t_{n}))$ with each $\chi(t_{k})=\zeta^{c_{k}}$ for $1\leq k\leq n$ and $1\leq c_{k}\leq r.$ Thus, we can identify each $\chi$ with $\mathrm{c}=(c_1,\ldots,c_{n})\in C^{n}$ such that $\chi(t_{k})=\zeta^{c_{k}}$ for $1\leq k\leq n$, where $C=\{1,\ldots,r\}.$

For each $1\leq k\leq r,$ we denote by $P_{k}(X)$ the Lagrange polynomial
\[P_{k}(X)=\prod_{1\leq l\leq r; l\neq k}\frac{X-\zeta^{l}}{\zeta^{k}-\zeta^{l}}.\]

The following lemma can be easily proved by definition.
\begin{lemma}
If we identify $\chi\in \mathrm{Irr}(\mathcal{T})$ with $\mathrm{c}\in C^{n}.$ Then we have
\[E_{\chi}=P_{c_{1}}(t_{1})\cdots P_{c_{n}}(t_{n}).\]
\end{lemma}

For each $\mathrm{c}=(c_1,\ldots,c_{n})\in C^{n}$, we define
\[L_{\mathrm{c}} :=P_{c_{1}}(w_{1})\cdots P_{c_{n}}(w_{n}).\]
The following proposition can be easily obtained by using [HNT, Lemma 4.1] and Proposition \ref{propo-presenta-3}.
\begin{proposition}\label{propo-presenta-4}
$Y_{r,n}(q)$ has a forth presentation, which is generated by $h_{1},\ldots,h_{n-1}$ and $L_{\mathrm{c}}$ $(\mathrm{c}\in C^{n})$ with relations$:$
\begin{align*}
\sum_{\mathrm{c}\in C^{n}}L_{\mathrm{c}}&=1; \\[0.1em]
L_{\mathrm{c}'}L_{\mathrm{c}}&=\delta_{\mathrm{c}', \mathrm{c}}L_{\mathrm{c}};\\[0.1em]
h_{i}h_{j}&=h_{j}h_{i}\qquad \quad\qquad\qquad\hspace{2.1mm}\mbox{for all $1\leq i, j\leq n-1$ such that $\vert i-j\vert > 1;$}\\[0.1em]
h_ih_{i+1}h_i&=h_{i+1}h_ih_{i+1}  \qquad \quad\quad \hspace{3mm} \mbox{for all $1\leq i\leq n-2;$}\\[0.1em]
h_{i}^{2}&=q+(q-1)h_{i} \qquad \qquad \hspace{-0.5mm}\mbox{for all $1\leq i\leq n-1;$}\\[0.1em]
h_{i}L_{\mathrm{c}}&=L_{s_{i}(\mathrm{c})}h_{i}-(q-1)
\begin{cases}
-L_{\mathrm{c}} & \text{if } c_{i}< c_{i+1},
\\
~~~0 & \text{if } c_{i}=c_{i+1},\\
L_{s_{i}(\mathrm{c})}& \text{if } c_{i}> c_{i+1},
\end{cases}
\end{align*}
where $s_{i}$ acts on $\mathrm{c}$ by exchanging $c_{i}$ and $c_{i+1}.$
\end{proposition}

By [HNT, Proposition 4.2], the following set
\begin{align}\label{basis-yokonuam-hecke-4}
\{L_{\mathrm{c}}h_{w}\:|\:\mathrm{c}\in C^{n}\text{ and }w\in\mathfrak{S}_{n}\}
\end{align}
also gives rise to a $\mathbb{K}$-basis of $Y_{r,n}(q).$

In this paper, we consider the particular case $Y_{r,n}(0)$ when $q=0,$ which we call $0$-Yokonuma-Hecke algebras.

\subsection{Classification of simple modules}
For each $\mathrm{c}=(c_1,\ldots,c_{n})\in C^{n},$ we define an associative set $I_{\mathrm{c}}=((I_{\mathrm{c}})_{1},\ldots,(I_{\mathrm{c}})_{k}),$ where $(I_{\mathrm{c}})_{j}=(c_{i_{j-1}+1},\ldots,c_{i_{j}})$ is such that $c_{i_{j-1}+1}=\ldots=c_{i_{j}}$ and $c_{i_{j}}\neq c_{i_{j}+1}$ for $1\leq j\leq k,$ and where $i_{0}=0$ and $i_{k}=n.$ We also define $|(I_{\mathrm{c}})_{j}|=i_{j}-i_{j-1}$ for $1\leq j\leq k.$
\begin{theorem}\label{classification-result1}
All the simple $Y_{r,n}(0)$-modules are of dimension $1.$ Moreover, they are indexed by the following set$:$
\[\big\{(\mathrm{c}, \mathrm{J})=((c_1,\ldots,c_{n}),(J_1,\ldots,J_{k}))\:|\:\mathrm{c}\in C^{n}\text{ and }J_{j}\text{ is a composition of }|(I_{\mathrm{c}})_{j}|\text{ for }1\leq j\leq k\big\}.\]
\end{theorem}
\begin{proof}
By definition, we can easily check that the following equalities hold.

(1) $(g_{i}g_{i+1}-g_{i+1}g_{i})^{3}=0;$

(2) $(g_{i}t_{i}-t_{i}g_{i})^{2}=(g_{i}t_{i+1}-t_{i+1}g_{i})^{2}=0.$

Using \eqref{relations-cude}, $(1)$ and $(2)$, it involves making a lengthy bur routine calculation to get that the commutators $[g_{i}, g_{i+1}]$ and $[g_{i}, t_{i}]$ are strongly nilpotent elements. Let $J$ denote the two-sided ideal generated by all the commutators $[g_{i}, g_{j}]$ and $[g_{i}, t_{j}].$ Thus, we have $J\subseteq \mathrm{rad}(Y_{r,n}(0)).$ Let $\overline{Y}_{r,n}(0)=Y_{r,n}(0)/J.$ Then $\overline{Y}_{r,n}(0)$ is a commutative algebra generated by $\overline{t}_{1},\ldots,\overline{t}_{n}$ and $\overline{g}_{1},\ldots,\overline{g}_{n-1}$ with relations:
\begin{align*}
\overline{t}_{i}^{r}=1,\quad \overline{t}_{i}\overline{t}_{j}=\overline{t}_{j}\overline{t}_{i},\quad \overline{g}_{i}\overline{t}_{j}=\overline{t}_{j}\overline{g}_{i},\quad \overline{g}_{i}\overline{t}_{i}=\overline{t}_{i+1}\overline{g}_{i},\quad  \overline{g}_{i}\overline{g}_{j}=\overline{g}_{j}\overline{g}_{i}\quad\text{ and }\quad \overline{g}_{i}^{2}=-\overline{g}_{i}.
\end{align*}
It is easy to see that $\overline{Y}_{r,n}(0)$ has no nilpotent elements. So $\overline{Y}_{r,n}(0)$ is semisimple and $J\supseteq \mathrm{rad}(Y_{r,n}(0)).$ Therefore, $J=\mathrm{rad}(Y_{r,n}(0)),$ and $\overline{Y}_{r,n}(0)=Y_{r,n}(0)/\mathrm{rad}(Y_{r,n}(0)).$

Since $\overline{Y}_{r,n}(0)$ is commutative, all the simple $Y_{r,n}(0)$-modules are of dimension $1,$ and moreover, they are indexed by the set $(\mathrm{c}, \mathrm{J}).$ In fact, for each $(\mathrm{c}, \mathrm{J})$ such that $\mathrm{c}\in C^{n}\text{ and }J_{j}\text{ is a composition of }|(I_{\mathrm{c}})_{j}|\text{ for }1\leq j\leq k,$ let $D(J_{i})$ be the associated subset of $[1, |(I_{\mathrm{c}})_{j}|-1].$ Then the associated irreducible representation $\varphi_{(\mathrm{c}, \mathrm{J})}$ of $Y_{r,n}(0)$ is defined by
\begin{align*}
\varphi_{(\mathrm{c}, \mathrm{J})}(t_{i})=\zeta^{c_{i}}\quad \text{ and }\quad \varphi_{(\mathrm{c}, \mathrm{J})}(g_{k})=\begin{cases}
-1 & \text{if }k\in \sum\limits_{1\leq l\leq j-1}|(I_{\mathrm{c}})_{l}|+D(J_{j}) \text{ for some }j,
\\
0 & \text{otherwise.}
\end{cases}
\end{align*}
\end{proof}

\subsection{$Y_{r,n}(0)$ is Frobenius}
Recall that an algebra $A$ over $\mathbb{K}$ is called Frobenius if there is a linear map $\tau :A\rightarrow \mathbb{K}$ whose kernel contains no non-zero left or right ideal of $A.$ The following proposition claims that $Y_{r,n}(0)$ is Frobenius, and is in particular self-injective.
\begin{proposition}\label{frobenius-algebras }
$Y_{r,n}(0)$ is a Frobenius algebra.
\end{proposition}
\begin{proof}
We define a linear map $\tau :Y_{r,n}(0)\rightarrow \mathbb{K}$ by
\begin{align}\label{frobenius-forms-cude}
\tau(t_{1}^{a_{1}}\cdots t_{n}^{a_{n}}g_{w})=\begin{cases}
1 & \text{if }w=w_{0},
\\
0 & \text{otherwise.}
\end{cases}
\end{align}
where $w_{0}$ is the longest element of $\mathfrak{S}_{n}.$ We then claim that for any $0\neq h\in Y_{r,n}(0)$, there exist elements $j$ and $k$ such that $\tau(jh)$ and $\tau(hk)$ are non-zero. By \eqref{basis-yokonuam-hecke-1}, we can write $h$ as a linear combination of some elements $g_{y}t_{1}^{b_{1}}\cdots t_{n}^{b_{n}}$, and assume that $w$ is an element of maximal length such that $g_{w}t_{1}^{c_{1}}\cdots t_{n}^{c_{n}}$ occurs with non-zero coefficient. Let $j=g_{w_{0}w^{-1}}$ and $k=g_{w^{-1}w_{0}}.$

Since we have
\begin{align*}
g_{i}g_{w}=\begin{cases}
g_{s_{i}w} & \text{if }\ell(s_{i}w)> \ell(w),
\\
-e_{i}g_{w} & \text{if }\ell(s_{i}w)< \ell(w),
\end{cases}
\end{align*}
we notice that for any $x, y\in \mathfrak{S}_{n},$ $g_{x}g_{y}$ is of the form $tg_{z}$ such that $\ell(z)\leq \ell(x)+\ell(y),$ and $\ell(z)=\ell(x)+\ell(y)$ if and only if $\ell(xy)=\ell(x)+\ell(y)$, in which case $z=xy$ and $t=1.$ Thus, we get that $j\cdot g_{w}t_{1}^{c_{1}}\cdots t_{n}^{c_{n}}=g_{w_{0}}t_{1}^{c_{1}}\cdots t_{n}^{c_{n}}$ and $g_{w}t_{1}^{c_{1}}\cdots t_{n}^{c_{n}}\cdot k=g_{w_{0}}t_{1}^{c_{1}'}\cdots t_{n}^{c_{n}'}$, while $\tau(j\cdot g_{x}t_{1}^{b_{1}}\cdots t_{n}^{b_{n}})=0=\tau(g_{x}t_{1}^{b_{1}}\cdots t_{n}^{b_{n}}\cdot k)$ for any $x\neq w$ with $\ell(x)\leq \ell(w).$
\end{proof}

There is an involution $\phi$ on $Y_{r,n}(0)$ defined by $\phi(g_{i})=g_{n-i}$ for $1\leq i\leq n-1$ and $\phi(t_{j})=t_{n+1-j}$ for $1\leq j\leq n.$ Then we have the following result.
\begin{proposition}\label{frobenius-algebras-cude}
Let $\tau :Y_{r,n}(0)\rightarrow \mathbb{K}$ be defined as in \eqref{frobenius-forms-cude}. Then for all $a$ and $b$ in $Y_{r,n}(0)$, we have $\tau(ab)=\tau(\phi(b)a).$
\end{proposition}

\subsection{$Y_{r,n}(0)$ is standardly based}
Recall that the notion of a standardly based algebra was introduced in [DR, Definition 1.2.1] and a complete classification of a finite dimensional standardly based algebra over a field is also provided in [DR, Theorem 2.4.1]. The proof of the following theorem is inspired by that of [YL, Theorem 5.1].
\begin{theorem}\label{standardly-based-algebras}
$Y_{r,n}(0)$ is a standardly based algebra with a standard basis $\{E_{\chi}g_{w}\:|\:\chi\in \mathrm{Irr}(\mathcal{T})\text{ and }w\in\mathfrak{S}_{n}\}.$ Moreover, the simple modules of $Y_{r,n}(0)$ over $\mathbb{K}$ are exactly those which are given in Theorem \ref{classification-result1}.
\end{theorem}
\begin{proof}
Since we have
\begin{align*}
g_{i}g_{w}=\begin{cases}
g_{s_{i}w} & \text{if }\ell(s_{i}w)> \ell(w),
\\
-\sum\limits_{\chi; s_{i}(\chi)=\chi}E_{\chi}g_{w} & \text{if }\ell(s_{i}w)< \ell(w),
\end{cases}
\end{align*}
we get that
\begin{align}\label{standard-basis-relation}
E_{\chi'}g_{i}\cdot E_{\chi}g_{w}=E_{\chi'}E_{s_{i}(\chi)}g_{i}g_{w}=\begin{cases}
0& \text{if } s_{i}(\chi)\neq \chi',\\
E_{\chi'}g_{s_{i}w} & \text{if }s_{i}(\chi)=\chi'\text{ and }\ell(s_{i}w)> \ell(w),
\\
-E_{\chi'}g_{w} & \text{if }s_{i}(\chi)=\chi'=\chi\text{ and }\ell(s_{i}w)< \ell(w),\\
0 & \text{if }s_{i}(\chi)=\chi'\neq\chi\text{ and }\ell(s_{i}w)< \ell(w).
\end{cases}
\end{align}

Set $\Lambda :=\{(\chi, w)\:|\:\chi\in \mathrm{Irr}(\mathcal{T})\text{ and }w\in\mathfrak{S}_{n}\}.$ We fix a total ordering on the elements of $\mathfrak{S}_{n}$ such that $w_{i}> w_{j}$ whenever $\ell(w_{i})> \ell(w_{j})$ and extend it to $\Lambda$. Assume that $I(\chi, w)$ and $J(\chi, w)$ consist of only one element $E_{\chi}g_{w}$ for each $(\chi, w)\in \Lambda.$ By \eqref{standard-basis-relation}, we see that the basis $\{E_{\chi}g_{w}\}$ is a standard basis.

In order to classify the simple modules of $Y_{r,n}(0)$ using [DR, Theorem 2.4.1], we must choose those $(\chi, w)$ such that $\beta_{(\chi, w)}(E_{\chi}g_{w}, E_{\chi}g_{w})\neq 0.$ It is easy to see that they exactly correspond to the elements contained in the set stated in Theorem \ref{classification-result1}, if we identify $\mathrm{Irr}(\mathcal{T})$ with $C^{n}$ and set $w$ to be the longest element of $W_{J}$, where $W_{J}$ is the Young subgroup of $\mathfrak{S}_{n}$ associated to $J.$
\end{proof}

\section{Appendix. Representations of nil Yokonuma-Hecke algebras}

In the appendix, we present two different approaches to classifying the simple modules of the nil Yokonuma-Hecke algebra ${}^{0}\!Y_{r,n}^{f}$ over an algebraically closed field $\mathbb{K}$ of characteristic $p$ such that $p$ does not divide $r.$

${}^{0}\!Y_{r,n}^{f}$ is a $\mathbb{K}$-associative algebra generated by the elements $T_{1},\ldots, T_{n-1}$ and $t_{1},\ldots, t_{n}$ with relations:
\begin{equation}\label{rel-def-nilY1}\begin{array}{rclcl}
t_i^r\hspace*{-7pt}&=&\hspace*{-7pt}1 && \mbox{for all $i=1,\ldots,n$;}\\[0.2em]
t_it_j\hspace*{-7pt}&=&\hspace*{-7pt}t_jt_i &&  \mbox{for all $i,j=1,\ldots,n$;}\\[0.1em]
T_it_j\hspace*{-7pt}&=&\hspace*{-7pt}t_{s_i(j)}T_i && \mbox{for all $i=1,\ldots,n-1$ and $j=1,\ldots,n$;}\\[0.1em]
T_iT_j\hspace*{-7pt}&=&\hspace*{-7pt}T_jT_i && \mbox{for all $i,j=1,\ldots,n-1$ such that $\vert i-j\vert \geq 2$;}\\[0.1em]
T_iT_{i+1}T_i\hspace*{-7pt}&=&\hspace*{-7pt}T_{i+1}T_iT_{i+1} && \mbox{for all $i=1,\ldots,n-2$;}\\[0.1em]
T_{i}^{2}\hspace*{-7pt}&=&\hspace*{-7pt}0 && \mbox{for all $i=1,\ldots,n-1$.}
\end{array}
\end{equation}

Let $w\in \mathfrak{S}_{n}$ be with a reduced expression $w=s_{i_1}\cdots s_{i_r}$. Then $T_{w} :=T_{i_{1}}\cdots T_{i_{r}}$ is independent of the choice of the reduced expression of $w.$ Let ${}^{0}\!H_{n}^{f}$ denote the subalgebra of ${}^{0}\!Y_{r,n}^{f}$ generated by $T_{1},\ldots, T_{n-1}$, which is exactly the nil Hecke algebra. Let $\mathbb{K}T$ denote the subalgebra generated by $t_{1},\ldots, t_{n}.$ It is easy to see that the elements
\begin{align}\label{basis-nilyokonuam-hecke-1}
\{t_{1}^{a_{1}}\cdots t_{n}^{a_{n}}T_{w}\:|\:0\leq a_{1},\ldots,a_{n}\leq r-1\text{ and }w\in\mathfrak{S}_{n}\}
\end{align}
form a $\mathbb{K}$-basis of ${}^{0}\!Y_{r,n}^{f}.$

The following proposition gives a classification of simple modules of ${}^{0}\!Y_{r,n}^{f}.$ 
\begin{proposition}\label{simple-modules-nilyhecke}
The Jacobson radical of ${}^{0}\!Y_{r,n}^{f}$ is the two-sided ideal generated by the elements $T_{w}$ for all $w\neq 1.$ Moreover, there are $r^{n}$ non-isomorphic simple modules of ${}^{0}\!Y_{r,n}^{f},$ which are all of dimension $1.$ 
\end{proposition}
\begin{proof}
Since $J^{n(n-1)+1}=0$ and ${}^{0}\!Y_{r,n}^{f}/J\simeq \mathbb{K}T$, we have $J=\mathrm{rad}({}^{0}\!Y_{r,n}^{f}).$ Moreover, all the simple modules of ${}^{0}\!Y_{r,n}^{f}$ are of dimension $1$, which are indexed by $\text{Irr}(T)$, where $\text{Irr}(T)$ denotes the set of characters of $\mathbb{K}T.$ In fact, for each $\chi\in \text{Irr}(T),$ the associated simple module $\eta_{\chi}$ is defined by 
\begin{align*}
\eta_{\chi}(t_{j})=\chi(t_{j}) \quad\text{ for }1\leq j\leq n\quad\text{ and }\quad \eta_{\chi}(T_{i})=0 \quad\text{ for }1\leq i\leq n-1.
\end{align*}

Since $t_{j}E_{\chi}T_{w_{0}}=\chi(t_{j})E_{\chi}T_{w_{0}}$ for $1\leq j\leq n$ and $T_{i}E_{\chi}T_{w_{0}}=0$ for $1\leq i\leq n-1.$ Thus, $\eta_{\chi}$ can also be realized as the minimal left ideal ${}^{0}\!Y_{r,n}^{f}E_{\chi}T_{w_{0}}$ in the regular representation. In fact, the $\mathbb{K}E_{\chi}T_{w_{0}},$ for $\chi\in \text{Irr}(T),$ give all the minimal non-zero two-sided ideals.
\end{proof}

There is an involution $\psi$ on ${}^{0}\!Y_{r,n}^{f}$ defined by $\phi(T_{i})=T_{n-i}$ for $1\leq i\leq n-1$ and $\phi(t_{j})=t_{n+1-j}$ for $1\leq j\leq n.$ Let $\lambda :{}^{0}\!Y_{r,n}^{f}\rightarrow \mathbb{K}$ be a linear map by $\tau(t_{1}^{a_{1}}\cdots t_{n}^{a_{n}}T_{w})=\delta_{w, w_{0}}.$ Then we have the following result.
\begin{proposition}\label{nil-yhecke-frobenius}
${}^{0}\!Y_{r,n}^{f}$ is a Frobenius algebra with a Frobenius form $\lambda$. Moreover, we have $\lambda(xy)=\lambda(\psi(y)x)$ for all $x, y\in {}^{0}\!Y_{r,n}^{f}.$
\end{proposition}

We also have the following result.
\begin{theorem}\label{standardly-based-algebrasnil-yhecke}
${}^{0}\!Y_{r,n}^{f}$ is a standardly based algebra with a standard basis $\{E_{\chi}T_{w}\:|\:\chi\in \mathrm{Irr}(T)\text{ and }w\in\mathfrak{S}_{n}\}.$ Moreover, the simple modules of ${}^{0}\!Y_{r,n}^{f}$ over $\mathbb{K}$ are exactly those which are given in Proposition \ref{simple-modules-nilyhecke}.
\end{theorem}
\begin{proof}
Since we have
\begin{align*}
T_{i}T_{w}=\begin{cases}
T_{s_{i}w} & \text{if }\ell(s_{i}w)> \ell(w),
\\
0 & \text{if }\ell(s_{i}w)< \ell(w),
\end{cases}
\end{align*}
we get that
\begin{align}\label{standard-basis-relation-nilyhecke}
E_{\chi'}T_{i}\cdot E_{\chi}T_{w}=E_{\chi'}E_{s_{i}(\chi)}T_{i}T_{w}=\begin{cases}
0& \text{if } s_{i}(\chi)\neq \chi',\\
E_{\chi'}T_{s_{i}w} & \text{if }s_{i}(\chi)=\chi'\text{ and }\ell(s_{i}w)> \ell(w),
\\
0 & \text{if }s_{i}(\chi)=\chi'\text{ and }\ell(s_{i}w)< \ell(w).
\end{cases}
\end{align}

Set $\Delta :=\{(\chi, w)\:|\:\chi\in \mathrm{Irr}(T)\text{ and }w\in\mathfrak{S}_{n}\}.$ We fix a total ordering on the elements of $\mathfrak{S}_{n}$ such that $w_{i}> w_{j}$ whenever $\ell(w_{i})> \ell(w_{j})$ and extend it to $\Delta$. Assume that $I(\chi, w)$ and $J(\chi, w)$ consist of only one element $E_{\chi}T_{w}$ for each $(\chi, w)\in \Delta.$ By \eqref{standard-basis-relation-nilyhecke}, we see that the basis $\{E_{\chi}T_{w}\}$ is a standard basis.

In order to classify the simple modules of ${}^{0}\!Y_{r,n}^{f}$ using [DR, Theorem 2.4.1], we must choose those $(\chi, w)$ such that $\beta_{(\chi, w)}(E_{\chi}g_{w}, E_{\chi}g_{w})\neq 0.$ Since $T_{w_{1}}T_{w_{2}}=T_{w_{1}w_{2}}$ if $\ell(w_{1}w_{2})=\ell(w_{1})+\ell(w_{2})$ and $T_{w_{1}}T_{w_{2}}=0$ otherwise, it is easy to see that $\beta_{(\chi, w)}(E_{\chi}T_{w}, E_{\chi}T_{w})\neq 0$ if and only if $w=1.$
\end{proof}

\noindent{\bf Acknowledgements.}
The author was partially supported by the National Natural Science Foundation of China (No. 11601273).
%/////////////////////////////////////////////////////////////////////////////////////////////////////////////////////////////////////////////////////////

%by (\ref{stn})
%\begin{align*}
%  t_{n+1}&\geq\frac{1}{t_n+(1-t_n)\frac{1}{c}}\\
%  1-t_{n+1}&\leq \frac{(1-t_n)(\frac{1}{c}-1)}{t_n+(1-t_n)\frac{1}{c}}=\frac{(1-t_n)(\frac{1}{c}-1)}{(1-t_n)(\frac{1}{c}-1)+1}\\
%  \frac{1}{1-t_{n+1}}&\geq \frac{1}{(1-t_n)(\frac{1}{c}-1)+1}
%\end{align*}
%Since $c\geq \frac{1}{2}$, $\frac{1}{c}-1\geq 1$, we have
%$$\frac{1}{1-t_{n+1}}\geq \frac{1}{1-t_n}+1\geq\cdots\geq\frac{1}{1-t_1}+n\geq n+1$$
%i.e. $1-t_{n+1}\leq \frac{1}{n+1}$
%$$0\leq v_{2n+1}-x^*\leq v_{2n+1}-v_{2n}\leq v_{2n+1}-t_nv_{2n+1}\geq \frac{1}{n}v_{2n+1}$$
%So $\|v_{2n+1}-x^*\|\leq \frac{N}{n}\|\bar{v}\|$, where $N$ is the normal constant of $P$.
%/////////////////////////////////////////////////////////////////////////////////////////////////////////////////////////////////

\end{document}